\documentclass[a4paper,10pt]{article}

\usepackage{graphicx}
\DeclareGraphicsExtensions{.pdf}
\graphicspath{{pics//}}

\usepackage{amsmath,amssymb}
\usepackage{enumerate,enumitem}
\usepackage{MnSymbol}
\usepackage{nth}

\usepackage{amsthm}

\theoremstyle{plain}
\newtheorem{theorem}{Theorem}
\newtheorem{lemma}[theorem]{Lemma}
\newtheorem{question}[theorem]{Question}
\newtheorem{proposition}[theorem]{Proposition}

\newcommand{\bo}{\ensuremath{\operatorname{box}}}
\newcommand{\barbo}{\ensuremath{\overline{\operatorname{box}}}}
\newcommand{\calC}{\ensuremath{\mathcal{C}}}
\newcommand{\calI}{\ensuremath{\mathcal{I}}}
\newcommand{\barC}{\ensuremath{\overline{\mathcal{C}}}}
\newcommand{\calG}{\ensuremath{\mathcal{G}}}

\newcommand{\N}{\ensuremath{\mathbb{N}}}
\newcommand{\R}{\ensuremath{\mathbb{R}}}

\title{Local and Union Boxicity}

\author{Thomas Bl\"asius, Peter Stumpf and Torsten Ueckerdt}

\begin{document}

\maketitle

\begin{abstract}

  The boxicity $\bo(H)$ of a graph $H$ is the smallest integer $d$
  such that $H$ is the intersection of $d$ interval graphs, or
  equivalently, that $H$ is the intersection graph of axis-aligned
  boxes in $\R^d$.
  These intersection representations can be interpreted as covering representations of the complement $H^c$ of $H$ with co-interval graphs, that is, complements of interval graphs.  
  We follow the recent framework of global, local and folded covering numbers
  (Knauer and Ueckerdt, \textit{Discrete Mathematics} \textbf{339} (2016)) to define two new
  parameters: the local boxicity $\bo_\ell(H)$ and the union boxicity
  $\barbo(H)$ of $H$.
  The union boxicity of $H$ is the smallest $d$ such that $H^c$ can be covered with $d$ vertex-disjoint unions of co-interval graphs, while the local boxicity of $H$ is the smallest $d$ such that $H^c$ can be covered with co-interval graphs, at most $d$ at every vertex.
 
 We show that for every graph $H$ we have $\bo_\ell(H) \leq \barbo(H) \leq \bo(H)$ and that each of these inequalities can be arbitrarily far apart.
 Moreover, we show that local and union boxicity are also characterized by intersection representations of appropriate axis-aligned boxes in $\R^d$.
 We demonstrate with a few striking examples, that in a sense, the local boxicity is a better indication for the complexity of a graph, than the classical boxicity.
\end{abstract}

\section{Introduction}\label{sec:introduction}

An \emph{interval graph} is an intersection graph of intervals on the real line\footnote{Throughout, we shall just say ``intervals'' and drop the suffix ``on the real line''.}.
Such a set $\{I(v) \subseteq \R \mid v \in V(H)\}$ of intervals with $vw \in E(H) \Leftrightarrow I(v) \cap I(w) \neq \emptyset$ is called an \emph{interval representation of $H$}.
A box in $\R^d$, also called a \emph{$d$-dimensional box}, is the Cartesian product of $d$ intervals.
The \emph{boxicity} of a graph $H$, denoted by $\bo(H)$, is the least integer $d$ such that $H$ is the intersection graph of $d$-dimensional boxes, and a corresponding set $\{B(v) \subseteq \R^d \mid v\in V(H)\}$ is a \emph{box representation of $H$}.
The boxicity was introduced by Roberts~\cite{boxicity-intro} in 1969 and has many applications in as diverse areas as ecology and operations research~\cite{CR-83}.

As two $d$-dimensional boxes intersect if and only if each of the $d$
corresponding pairs of intervals intersect, we have the following more
graph theoretic interpretation of the boxicity of a graph; also see Figure~\ref{fig:covering}(a).


\begin{theorem}[Roberts~\cite{boxicity-intro}]
 For a graph $H$ we have $\bo(H) \leq d$ if and only if $H = G_1 \cap \cdots \cap G_d$ for some interval graphs $G_1,\ldots,G_d$.
\end{theorem}

I.e., the boxicity of a graph $H$ is the least integer $d$ such that $H$ is the intersection of some $d$ interval graphs.
For a graph $H = (V,E)$ we denote its complement by $H^c = (V, \binom{V}{2} - E)$.
Then by De Morgan's law we have
\begin{equation}
 H = G_1 \cap \cdots \cap G_d \quad \Longleftrightarrow \quad H^c = G_1^c \cup \cdots \cup G_d^c,\label{eq:DeMorgan}
\end{equation}
i.e., $\bo(H)$ is the least integer $d$ such that the complement $H^c$ of $H$ is the union of $d$ co-interval graphs $G_1^c,\ldots,G_d^c$, where a \emph{co-interval graph} is the complement of an interval graph\footnote{Equivalently, these are the comparability graphs of interval orders.}.
In other words, $\bo(H) \leq d$ if $H^c$ can be covered with $d$ co-interval graphs.
Strictly speaking, we have to be a little more precise here.
In order to use De Morgan's law, we should guarantee that $G_1,\ldots,G_d$ in~\eqref{eq:DeMorgan} all have the same vertex set.
To this end, if $G$ is a subgraph of $H$, let $\bar{G} = (V(H),E(G))$ be the graph obtained from $G$ by adding all vertices in $V(H) - V(G)$ as isolated vertices.
(Whenever we use $\bar{G}$ it will be clear from the context which supergraph $H$ of $G$ we consider.)
Clearly we have
\[
 H^c = G_1^c \cup \cdots \cup G_d^c \quad \Rightarrow \quad H^c = \bar{G}_1^c \cup \cdots \cup \bar{G}_d^c \quad \Rightarrow \quad H = \bar{G}_1 \cap \cdots \cap \bar{G}_d
\]
for any graph $H$ and any subgraphs $G_1,\ldots,G_d$ of $H$.
Now whenever $G$ is a co-interval graph, then so is $\bar{G}$, implying that $\bo(H)$ is the least integer $d$ such that $H^c$ can be covered with $d$ co-interval graphs.

\paragraph{Graph covering parameters.}
In the general graph covering problem one is given an input graph $H$, a so-called covering class $\calG$ and a notion of how to cover $H$ with one or more graphs from $\calG$.
The most classic notion of covering, which also corresponds to the boxicity as discussed above, is that $H$ shall be the union of $G_1,\ldots,G_t \in \calG$, i.e., $V(H) = \bigcup_{i \in [t]} V(G_i)$ and $E(H) = \bigcup_{i \in [t]} E(G_i)$.
(Here and throughout the paper, for a positive integer $t$ we denote $[t] = \{1,\ldots,t\}$.)
The \emph{global covering number}, denoted by $c_g^{\calG}(H)$, is then defined to be the minimum $t$ for which such a cover exists.
Many important graph parameters can be interpreted as a global covering number, e.g., the arboricity~\cite{Nas-64}, the track number~\cite{Gya-95} (this is not the track-number as defined in~\cite{DMW-05}) and the thickness~\cite{Bei-69,Mut-98}, just to name a few.

Most recently, Knauer and Ueckerdt~\cite{KU16} suggested the following unifying framework for three kinds of covering numbers, differing in the underlying notion of covering.
A \emph{graph homomor\-phism} is a map $\varphi: V(G) \to V(H)$ with the property that if $uv \in E(G)$ then $\varphi(u)\varphi(v) \in E(H)$, i.e., $\varphi$ maps vertices of $G$ (not necessarily injectively) to vertices of $H$ such that edges are mapped to edges.
For abbreviation we shall simply write $\varphi : G \to H$ instead of $\varphi: V(G) \to V(H)$.
For an input graph $H$, a covering class $\calG$ and a positive integer $t$, a \emph{$t$-global $\calG$-cover of $H$} is an edge-surjective homomor\-phism $\varphi: G_1 \cupdot \cdots \cupdot G_t \to H$ such that $G_i \in \calG$ for each $i \in [t]$.
Here $\cupdot$ denotes the vertex-disjoint union of graphs.
We say that $\varphi$ is \emph{injective} if its restriction to $G_i$ is injective for each $i \in [t]$.
A $\calG$-cover is called \emph{$s$-local} if $|\varphi^{-1}(v)| \leq s$ for every $v \in V(H)$.

Hence, if $\varphi$ is a $\calG$-cover of $H$, then
\begin{itemize}[label=]
 \item $\varphi$ is $t$-global if it uses only $t$ graphs from the covering class $\calG$, 
 \item $\varphi$ is injective if $\varphi(G_i)$ is a copy of $G_i$ in $H$ for each $i \in [t]$,
 \item $\varphi$ is $s$-local if for each $v \in V(H)$ at most $s$ vertices are mapped onto $v$.
\end{itemize}

\begin{figure}[tb]
  \centering
  \includegraphics[page=2]{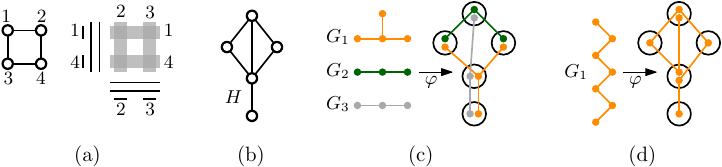}
  \caption{(a)~The 4-cycle as the intersection of two interval graphs.
    (b)~Example graph~$H$.  (c)~An injective covering of $H$ that is
    3-global and 2-local.  (d)~A (non-injective) 1-global 2-local
    covering of $H$.}
  \label{fig:covering}
\end{figure}

For a covering class $\calG$ and an input graph $H$ the \emph{global
  covering number} $c_{g}^{\calG}(H)$, the \emph{local covering
  number} $c_{\ell}^{\calG}(H)$, and the \emph{folded covering number}
$c_{f}^{\calG}(H)$ are then defined as follows; see also
Figure~\ref{fig:covering}(b)--(d):
 \begin{itemize}[label=]
  \item $c_{g}^{\calG}(H) = \min\left\{ t : \text{there exists a $t$-global injective }\calG\text{-cover of }H\right\}$
  \item $c_{\ell}^{\calG}(H) = \min\left\{ s : \text{there exists an $s$-local injective }\calG\text{-cover of }H\right\}$
  \item $c_{f}^{\calG}(H) = \min\left\{ s : \text{there exists a $1$-global $s$-local }\calG\text{-cover of }H\right\}$
 \end{itemize}

Intuitively speaking, for $c_\ell^{\calG}(H)$ we want to represent the input graph $H$ as the union of graphs from the covering class $\calG$, where the number of graphs we use is not important.
Rather we want to ``use'' each vertex of $H$ in only few of these subgraphs.
For $c_f^{\calG}(H)$ it is convenient to think of the ``inverse'' mapping for $\varphi$.
If $\varphi: G_1 \to H$ is a $1$-global $\calG$-cover of $H$, then the preimage under $\varphi$ of a vertex $v \in V(H)$ is an independent set $S_v$ in $G_1$.
Moreover, for every $u,v \in V(H)$ we have $uv \in E(H)$ if and only if there is at least one edge between $S_u$ and $S_v$ in $G_1$.
So $G_1$ is obtained from $H$ by a series of \emph{vertex splits}, where splitting a vertex $v$ into an independent set $S_v$ is such that for each edge $vw$ incident to $v$ there is at least one edge between $w$ and $S_v$ after the split.
Now $c_f^{\calG}(H)$ is the smallest $s$ such that each vertex can be split into at most $s$ vertices so that the resulting graph $G_1$ lies in the covering class $\calG$.

It is known, that if the covering class $\calG$ is closed under certain graph operations, we can deduce inequalities between the folded, local and global covering numbers.
For a graph class $\calG$ we define the following.

\begin{itemize}
 \item $\calG$ is \emph{homomor\-phism-closed} if for any connected $G \in \calG$ and any homomor\-phism $\varphi:G \to H$ into some graph $H$ we have that $\varphi(G) \in \calG$.
 \item $\calG$ is \emph{hereditary} if for any $G \in \calG$ and any induced subgraph $G'$ of $G$ we have that $G' \in \calG$.
 \item $\calG$ is \emph{union-closed} if for any $G_1,G_2 \in \calG$ we have that $G_1 \cupdot G_2 \in \calG$.
\end{itemize}


\begin{proposition}[Knauer-Ueckerdt~\cite{KU16}]\label{prop:covering-numbers-inequalities}
 For every input graph $H$ and every covering class $\calG$ we have
 \begin{enumerate}[label=(\roman*)]
  \item $c_\ell^{\calG}(H) \leq c_g^{\calG}(H)$, and if $\calG$ is union-closed, then $c_f^{\calG}(H) \leq c_\ell^{\calG}(H)$,
  
  \item if $G$ is hereditary and homomor\-phism-closed, then $c_f^{\calG}(H) \geq c_\ell^{\calG}(H)$.
 \end{enumerate}
\end{proposition}

\paragraph{Boxicity variants.}
Let us put the boxicity into the graph covering framework by Knauer and Ueckerdt~\cite{KU16} as described above.
To this end, let $\calC$ denote the class of all co-interval graphs.
Then we have $\bo(H) = c_g^{\calC}(H^c)$ and we can investigate the new parameters
\[
 \bo_f(H) := c_f^{\calC}(H^c) \quad \text{and} \quad \bo_\ell(H) := c_\ell^{\calC}(H^c).
\]
Clearly, if $H$ is an interval graph, i.e., $H^c \in \calC$, then $\bo_f(H) = \bo_\ell(H) = \bo(H) = 1$.
As it turns out, if $H$ is not an interval graph, then $\bo_f(H)$ is not very meaningful.

\begin{theorem}\label{thm:box-f-meaningless}
 For every graph $H$ we have $\bo_f(H) = 1$ if $H^c \in \calC$ and $\bo_f(H) = \infty$ otherwise.
\end{theorem}

Basically, Theorem~\ref{thm:box-f-meaningless} says that if $H^c$ is not a co-interval graph, there is no way to obtain a co-interval graph from $H^c$ by vertex splits.
For example, if $H$ has an induced $4$-cycle and hence $H^c$ has two independent edges, then $H^c \notin \calC$ and whatever vertex splits are applied, the result will always have two independent edges, i.e., not be a co-interval graph.
To overcome this issue, it makes sense to define $\barC$ to be the class of all vertex-disjoint unions of co-interval graphs and consider the parameters
\[
 \barbo(H) := c_g^{\barC}(H^c), \quad \barbo_\ell(H) := c_\ell^{\barC}(H^c), \quad \barbo_f(H) := c_f^{\barC}(H^c).
\]

We have defined in total six boxicity-related graph parameters, one of which (namely $\bo_f(H)$) turned out to be meaningless by Theorem~\ref{thm:box-f-meaningless}.
Somehow luckily, three of the remaining five parameters always coincide.

\begin{theorem}\label{thm:three-coincide}
 For every graph $H$ we have $\bo_\ell(H) = \barbo_\ell(H) = \barbo_f(H)$.
\end{theorem}

Proposition~\ref{prop:covering-numbers-inequalities} gives $\barbo_\ell(H) = c_\ell^{\barC}(H^c) \leq c_g^{\barC}(H^c) = \barbo(H)$ for every input graph $H$.
As $\calC \subset \barC$ we have $\barbo(H) = c_g^{\barC}(H^c) \leq c_g^{\calC}(H^c) = \bo(H)$ for every input graph $H$.
Hence with Theorem~\ref{thm:three-coincide} for every graph $H$ the remaining three boxicity-related parameters fulfil:

\begin{equation}
 \bo_\ell(H) \leq \barbo(H) \leq \bo(H).\label{eq:inequalities}
\end{equation}

We refer to $\bo_\ell(H)$ as the \emph{local boxicity of $H$} and to $\barbo(H)$ as the \emph{union boxicity of $H$}. 
Indeed, the three parameters boxicity, local boxicity and union boxicity are non-trivial and reflect different aspects of the graph, as will be investigated in more detail in this paper.

\begin{theorem}\label{thm:separation}
 For every positive integer $k$ there exist graphs $H_k, H'_k, H''_k$ with
 \begin{enumerate}[label = (\roman*)]
  \item $\bo_\ell(H_k) \geq k$,\label{enum:local-high}
  
  \item $\bo_\ell(H'_k) = 2$ and $\barbo(H'_k) \geq k$,\label{enum:local-smaller-union}
  
  \item $\barbo(H''_k) = 1$ and $\bo(H''_k) = k$.\label{enum:union-smaller-global}
 \end{enumerate}
\end{theorem}

We also give geometric interpretations of the local and union boxicity of a graph $H$ in terms of intersecting high-dimensional boxes.
For positive integers $k,d$ with $k \leq d$ we call a $d$-dimensional box $B = I_1 \times \cdots \times I_d$ \emph{$k$-local} if for at most $k$ indices $i \in \{1,\ldots,d\}$ we have $I_i \neq \R$.
Thus a $k$-local $d$-dimensional box is the Cartesian product of $d$ intervals, at least $d-k$ of which are equal to the entire real line $\R$.

\begin{theorem}\label{thm:geometric-interpretation}
 Let $H$ be a graph.
 \begin{enumerate}[label = (\roman*)]
 \item We have $\barbo(H) \leq k$ if and only if there exist
   $d_1,\ldots,d_k$ such that $H$ is the intersection graph of
   Cartesian products of $k$ boxes, where the $i$th box is $1$-local
   $d_i$-dimensional,
   $i = 1,\ldots,k$.\label{enum:interpretation-union}
  
  \item We have $\bo_\ell(H) \leq k$ if and only if there exists some $d$ such that $H$ is the intersection graph of $k$-local $d$-dimensional boxes.\label{enum:interpretation-local}
 \end{enumerate}
\end{theorem}

There is a number of results in the literature stating that the boxicity of certain graphs is low, for which we can easily see that the local boxicity is even lower.
Indeed, often an intersection representation with $d$-dimensional boxes is constructed, in order to show that $\bo(H) \leq d$, and in many cases these representations consist of $s$-local $d$-dimensional boxes for some $s < d$ (or can be turned into such quite easily).
Hence, with Theorem~\ref{thm:geometric-interpretation} we can conclude in such cases that $\bo_\ell(H) \leq s$.

Let us restrict here to one such case, which is comparably simple.
For a graph $H$ the \emph{acyclic chromatic number}, denoted by $\chi_a(H)$, is the smallest $k$ such that there exists a proper vertex coloring of $H$ with $k$ colors in which any two color classes induce a forest.
In other words, an acyclic coloring has no monochromatic edges and no bicolored cycles.
Esperet and Joret~\cite{EJ-13} have recently shown that for any graph $H$ with $\chi_a(H) = k$ we have $\bo(H) \leq k(k-1)$.
Indeed, their proof (which we include here for completeness) gives an intersection representation of $H$ with $2(k-1)$-local $k(k-1)$-dimensional boxes, implying the following theorem.

\begin{theorem}\label{thm:acyclic-coloring}
 For every graph $H$ we have $\bo_\ell(H) \leq 2(\chi_a(H)-1)$.
\end{theorem}
\begin{proof}
 Let $c$ be an acyclic coloring of $H$ with $k$ colors.
 For any pair $\{i,j\}$ of colors consider the subgraph $G_{i,j}$ induced by the vertices of colors $i$ and $j$.
 As $G_{i,j}$ is a forest, we have $\bo(G_{i,j}) \leq 2$ (this follows from~\cite{Sch-84} but can also be seen fairly easily).
 Moreover, since $H$ is the union of all $G_{i,j}$, the complement $H^c$ of $H$ is the intersection of the complements of all $\bar{G}_{i,j}$ (note the use of $\bar{G}_{i,j}$ instead of $G_{i,j}$ here).
 
 Now take an intersection representation of $G_{i,j}$ with $2$-dimensional boxes and extend it to one for $\bar{G}_{i,j}$ by putting the box $\R^2$ for each vertex colored neither $i$ nor $j$.
 Then the Cartesian product of all these $\binom{k}{2}$ box representations is an intersection representation of $H$ with $2(k-1)$-local $k(k-1)$-dimensional boxes.
 This proves that $\bo(H) \leq k(k-1)$ and $\bo_\ell(H) \leq 2(k-1)$, as desired.
\end{proof}

\paragraph{Organization of the paper.}

In Section~\ref{sec:general} we prove Theorem~\ref{thm:box-f-meaningless}, i.e., that $\bo_f(H)$ is meaningless, and Theorem~\ref{thm:three-coincide}, i.e., that three of the remaining five boxicity variants coincide.
In Section~\ref{sec:separation} we consider the problem of separation for boxicity and its local and union variants, that is, we give a proof of Theorem~\ref{thm:separation}.
In Section~\ref{sec:geometry} we describe and prove the geometric interpretations of local and union boxicity from Theorem~\ref{thm:geometric-interpretation}.
Finally, we give some concluding remarks and open problems in Section~\ref{sec:conclusions}.

\section{Local and Union Boxicity}\label{sec:general}

Recall that a graph class $\calG$ is homomor\-phism-closed if for every \emph{connected} graph $G \in \calG$ and any homorphism $\varphi : G \to H$ into some graph $H$ we have $\varphi(G) \in \calG$.
Since $\varphi$ is a homomor\-phism, $\varphi(G)$ arises from $G$ by a series of ``inverse vertex splits'', i.e., an independent set in $G$ is identified into a single vertex of $\varphi(G)$.
If $\calG$ is not only homomor\-phism-closed, but also closed under identifying non-adjacent vertices in \emph{disconnected} graphs, then the folded covering number $c_f^{\calG}$ turns out to be somewhat meaningless.

\begin{lemma}\label{lem:identifying-non-adjacent-vertices}
 If a covering class $\calG$ is closed under identifying non-adjacent vertices, then for every non-empty input graph $H$ we have
 \[
  c_f^{\calG}(H)<\infty \quad \Longleftrightarrow \quad H\in\calG \quad \Longleftrightarrow \quad c_f^{\calG}(H)=1.
 \]
\end{lemma}
\begin{proof}
 The right equivalence follows by definition of $c_f^{\calG}(H)$.
 
 The implication $H \in \calG \Rightarrow c_f^{\calG}(H) < \infty$ in the first equivalence is thereby obvious, and it is left to show that $c_f^{\calG}(H) = 1$ whenever $c_f^{\calG}(H) < \infty$.
 So let $\varphi:G_1 \to H$ be any $1$-global cover of $H$.
 We do induction over $|V(G_1)|$, the number of vertices in $G_1$.

 If $|V(G_1)| = |V(H)|$, i.e., no vertices are folded, then $\varphi$ is injective and therefore $c_f^{\calG}(H)=1$.
 So assume that $|V(G_1)| > |V(H)|$ and let $v,w$ be distinct vertices in $G_1$ with $\varphi(v) = \varphi(w)$.
 Consider the graph $G'_1$ that we obtain by identifying $v$ and $w$ in $G_1$.
 Since $\varphi(v) = \varphi(w)$ is only possible if $v$ and $w$ are non-adjacent, and $\calG$ is closed under identifying non-adjacent vertices we know that $G_1'\in\calG$.
 Now the $1$-global $\calG$-cover $\varphi : G_1 \to H$ induces a $1$-global $\calG$-cover $\varphi':G_1'\to H$ by $\varphi = \varphi' \circ \psi$, where $\psi : G_1 \to G'_1$ identifies $v$ and $w$ in $G_1$ and fixes all other vertices.
 As $|V(G'_1)| = |V(G_1)| - 1$, we can apply induction to $\varphi'$ to conclude that $c_f^{\calG}(H)=1$.
\end{proof}

\begin{lemma}\label{lem:class-properties}
 Let $\calC$ be the class of all co-interval graphs and $\barC$ be the class of all vertex-disjoint unions of co-interval graphs.
 Then
 \begin{enumerate}[label = (\roman*)]
  \item $\calC$ and $\barC$ are hereditary,\label{enum:barC-and-calC}
  \item $\calC$ is closed under identifying non-adjacent vertices, and\label{enum:calC}
  \item $\barC$ is homomor\-phism-closed.\label{enum:barC}
 \end{enumerate}
\end{lemma}
\begin{proof}
 \begin{enumerate}[label = (\roman*)]
  \item Consider any graph $G \in \barC$.
   Then $G = G_1 \cupdot \cdots \cupdot G_t$ for some $G_1,\ldots,G_t \in \calC$.
   If $G \in \calC$, then $t=1$.
   For $i \in [t]$ consider an intersection representation $\{I_i(v) \mid v \in V(G_i)\}$ of $G_i^c$ with intervals.
   For any vertex set $S \subseteq V(G)$, consider the induced subgraphs when restricted to vertices in $S$, i.e., $G' = G[S]$ and $G'_i = G_i[V(G_i) \cap S]$ for $i \in [t]$.
   Note that $\{I_i(v) \mid v \in V(G_i) \cap S\}$ is an interval representation of $(G'_i)^c$, i.e., $G'_i \in \calC$.
   Hence $G' = G'_1 \cupdot \cdots \cupdot G'_t \in \barC$ and $G' \in \barC$ if $t = 1$.
   This shows that $\calC$ and $\barC$ are hereditary.
 
  \item Let $G \in \calC$, $x,y$ be two non-adjacent vertices in $G$ and $\{I(v) \mid v \in V(G)\}$ be an intersection representation of $G$ with intervals.
   Let $G'$ be the graph obtained from $G$ by identifying $x$ and $y$ into a single vertex $z$.
%
   Since $xy \in E(G^c)$ we have $I(x) \cap I(y) \neq \emptyset$ and hence $I(z) := I(x) \cap I(y)$ is a non-empty interval.
   As for any interval $J$ we have $J \cap I(z) \neq \emptyset$ if and only if $J \cap I(x) \neq \emptyset$ or $J \cap I(y) \neq \emptyset$ or both, we have that $\{I(v) \mid v \in V(G), v \neq x,y\} \cup \{I(z)\}$ is an intersection representation of $(G')^c$ and thus $G' \in \calC$, as desired.

  \item If $G \in \barC$ then $G = G_1 \cupdot \cdots \cupdot G_t$ for some $G_1,\ldots,G_t \in \calC$.
   If $x,y$ are two non-adjacent vertices in the same connected component, then $x,y$ are in the same $G_i$, say $G_1$.
   By~\ref{enum:calC} identifying $x$ and $y$ in $G_1$ gives a graph $G'_1 \in \calC$.
   Moreover, identifying $x$ and $y$ in $G$ gives a graph $G' = G'_1 \cupdot G_2 \cupdot \cdots \cupdot G_t$.
   As $G'_1 \in \calC$ we have $G' \in \barC$ and hence $\barC$ is homomor\-phism-closed.\qedhere
 \end{enumerate}
\end{proof}

%
\begin{proof}[Proof of Theorem~\ref{thm:box-f-meaningless}]
 This is a direct corollary of Lemma~\ref{lem:identifying-non-adjacent-vertices} and Lemma~\ref{lem:class-properties}~\ref{enum:calC}.
\end{proof}

%
\begin{proof}[Proof of Theorem~\ref{thm:three-coincide}]
 We have that $\barC$ is hereditary by Lemma~\ref{lem:class-properties}~\ref{enum:barC-and-calC}, homomor\-phism-closed by Lemma~\ref{lem:class-properties}~\ref{enum:barC} and union-closed by definition.
 Hence by Proposition~\ref{prop:covering-numbers-inequalities} we have $\barbo_f(H) = c_f^{\barC}(H^c) = c_\ell^{\barC}(H^c) = \barbo_\ell(H)$.

 As $\calC \subset \barC$ we clearly have $\barbo_\ell(H) = c_\ell^{\barC}(H^c) \leq c_\ell^{\calC}(H^c) = \bo_\ell(G)$.
 Finally, consider any $s$-local $t$-global $\barC$-cover $\varphi : G_1 \cupdot \cdots \cupdot G_t \to H^c$.
 For $i = 1,\ldots,t$ we have $G_i \in \barC$ and hence $G_i$ is the vertex-disjoint union of some graphs in $\calC$.
 Thus we can interpret $\varphi$ as an $s$-local $t'$-global $\calC$-cover of $H^c$ for some $t' \geq t$.
 This shows that $\bo_\ell(H) = c_\ell^{\calC}(H^c) \leq c_\ell^{\barC}(H^c) = \barbo_\ell(H)$ and thus concludes the proof.
\end{proof}

\section{Separating the Variants}\label{sec:separation}

%
%
%
\begin{proof}[Proof of Theorem~\ref{thm:separation}]{\ }
 \begin{enumerate}[label = (\roman*)]
 
  \item For a fixed integer $k \geq 1$ we consider any graph $F_k$
    that is $2k$-regular and has girth at least~$6$ (i.e., its
    shortest cycle has length at least~$6$).
   Now let $\varphi$ be an injective $s$-local $\calC$-cover of $F_k$, i.e., a cover of $E(F_k)$ with $t$ co-interval graphs $G_1,\ldots,G_t \subseteq F_k$ for some $t \in \N$ such that every vertex of $F_k$ is contained in at most $s$ such $G_i$.
   We shall show that $s \geq k$, proving that $c_\ell^\calC(F_k) \geq k$ and hence $\bo_\ell(H_k) \geq k$, where $H_k = F_k^c$ denotes the complement of $F_k$.
   
   A co-interval graph $G$ does not contain any induced matching on two edges.
   Hence $G$ does not contain any induced cycle of length at least~$6$.
   (Moreover, as $G$ is perfect, it also contains no induced cycles of length~$5$.)
   Since $F_k$ has girth at least~$6$, this implies that every subgraph of $F_k$ that is a co-interval graph is a forest.
   In particular, every $G_i$ has average degree less than~$2$, i.e., $\sum_{v \in V(G_i)}\deg_{G_i}(v) < 2|V(G_i)|$.
   We conclude that
   \begin{align*}
    2k \cdot |V(F_k)| &= \sum_{v \in V(F_k)} \deg_{F_k}(v) \leq \sum_{v \in V(F_k)} \sum_{\substack{i \in [t]\\v \in V(G_i)}} \deg_{G_i}(v)\\
     &= \sum_{i = 1}^t \sum_{v \in V(G_i)} \deg_{G_i}(v) < \sum_{i=1}^t 2|V(G_i)| \leq 2s \cdot |V(F_k)|,
   \end{align*}
   where the first inequality holds since every edge of $F_k$ is covered and the last inequality holds since every vertex is contained in at most $s$ of the $G_i$, $i = 1,\ldots,t$.
   From the above it follows that $s \geq k$, as desired.
  
  \item Our proof follows the ideas of Milans~\textit{et al.}~\cite{MSW-15}, who consider $L(K_n)$, the line graph of $K_n$, and prove that $c_g^\calI(L(K_n)) \to \infty$ for $n \to \infty$, while $c_\ell^\calI(L(K_n)) = 2$ for every $n \in \N$, where $\calI$ denotes the class of all interval graphs.
   However, instead of using the ordered Ramsey numbers (which is also possible in our case) we shall rather use the following hypergraph Ramsey numbers:
   Let $K_n^3$, $n \in \N$, denote the complete $3$-uniform hypergraph on $n$ vertices, i.e., $K_n^3 = ([n],\binom{[n]}{3})$.
   For an integer $k \geq 1$, the Ramsey number $R_k(K^3_6)$ is the smallest integer $n$ such that every coloring of the hyperedges of $K_n^3$ with $k$ colors contains a monochromatic copy of $K_6^3$.
   The hypergraph Ramsey theorem implies that $R_k(K^3_6)$ exists for every $k$~\cite{Ram-30}.
  
   Now for fixed $k \geq 1$, choose an integer $n = n(k) > R_k(K^3_6)$ and consider $L(K_n)$, the line graph of $K_n$.
   Let $\varphi$ be any injective $t$-global $\barC$-cover of $L(K_n)$ with co-interval graphs $G_1,\ldots,G_t \subseteq L(K_n)$ for some $t \in \N$.
   We shall show that $t > k$, proving that $c_g^{\barC}(L(K_n)) > k$ and hence $\barbo(H'_k) > k$, where $H'_k = (L(K_n))^c$ denotes the complement of $L(K_n)$.
   
   Assume for the sake of contradiction that $t \leq k$.
   From the $\barC$-cover $\varphi$ of $L(K_n)$, we define a coloring $c$ of $E(K_n^3)$ with $t$ colors.
   Given $x,y,z \in [n]$ with $x<y<z$, let $c(x,y,z) = \min\{i \in [t] \mid \{xy,yz\} \in E(G_i)\}$ be the smallest index of a co-interval graph in $\{G_1,\ldots,G_t\}$ that covers the edge between $xy$ and $yz$ in $L(K_n)$.
   Since $n > R_k(K^3_6) \geq R_t(K^3_6)$ under $c$ there is a monochromatic copy of $K_6^3$, say it is in color $i$ and that its vertices are $\{x_1,\ldots,x_6\}$.
   This means that $G_i$ has a connected component containing $x_1,\ldots,x_6$ and in particular the edges $\{x_1x_2,x_2x_3\}$ and $\{x_4x_5,x_5x_6\}$ of $L(K_n)$.
   However, these two edges induce a matching in $L(K_n)$ and hence also in that connected component of $G_i$.
   This is a contradiction to that component being a co-interval graph, and thus implies that $t > k$, as desired.

   Finally, observe that for any $n \in \N$ the following is an injective $2$-local $\calC$-cover of $L(K_n)$:
   For each $i \in [n]$ let $G_i$ be the clique in $L(K_n)$ formed by all edges incident to vertex $i$ of $K_n$.
   Then $\{G_1,\ldots,G_n\}$ is a set of $n$ co-interval graphs in $L(K_n)$ with the property that every edge of $L(K_n)$ lies in exactly one $G_i$ and every vertex of $L(K_n)$ lies in exactly two $G_i$.
   This shows that $c_\ell^{\calC}(L(K_n)) = \bo_\ell(H'_k) \leq 2$.
  
  \item For fixed $k \geq 1$ consider $M_k$ the matching on $k$ edges.
   We shall show that $c_g^{\barC}(M_k) = 1$ and $c_g^{\calC}(M_k) = k$, proving that $\barbo(H''_k)=1$ and $\bo(H''_k)=k$, where $H''_k = M_k^c$ is the complement of $M_k$.
   Indeed, as every co-interval graph has at most one component containing an edge, any $\calC$-cover of $M_k$ contains at least $k$ co-interval graphs to cover all $k$ components of $M_k$.
   Since $K_2$ is a co-interval graph, there actually is an injective $k$-global $\calC$-cover of $M_k$.
   Thus, we have $c_g^{\calC}(M_k)=\bo(H''_k)=k$.
   
   On the other hand, the class $\barC$ is union-closed and, since $K_2$ is a co-interval graph, $\barC$ contains all matchings.
   In particular $M_k \in \barC$ and therefore we have $c_g^{\barC}(M_k) = \barbo(H''_k)=1$.\qedhere
 \end{enumerate}
\end{proof}

\section{Geometric Interpretations}\label{sec:geometry}

%

\begin{lemma}\label{lem:union-interpretation}
 A graph $H$ is the intersection graph of $1$-local $d$-dimensional boxes if and only if $H^c$ is the vertex-disjoint union of $d$ co-interval graphs.
\end{lemma}

\begin{figure}[tb]
  \centering
  \includegraphics[page=1]{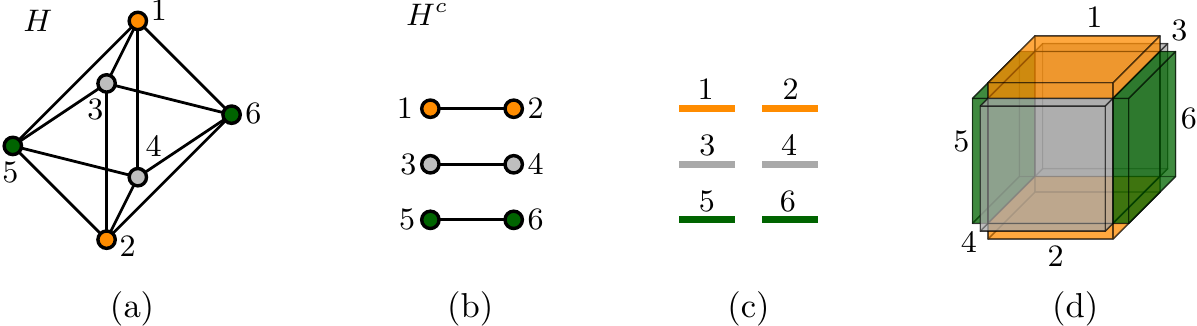}
  \caption{(a)~The octahedron $H$.  (b)~Its complement $H^c$.
    (c)~$H^c$ as the vertex-disjoint union of three co-interval graphs (given
    in their interval representation).  (d)~The corresponding
    intersection representation of $H$ with 1-local 3-dimensional
    boxes.  The two long sides of each box have actually infinite
    length.}
  \label{fig:geometric-interpretation}
\end{figure}

\begin{proof}
  For an illustration of the proof, see
  Figure~\ref{fig:geometric-interpretation}.  First, if
  $\{B(v) \mid v \in V(H)\}$ is an intersection representation of $H$
  with $1$-local boxes in $\R^d$, then for each $v \in V(H)$ let
  $B(v) = I_1(v) \times \cdots \times I_d(v)$.
  Without loss of generality assume that for every $v \in V(H)$ there is some coordinate $i \in [d]$ for which $I_i(v) \neq \R$.
  For each $i \in [d]$ consider the set
  $V_i = \{v \in V(H) \mid I_i(v) \neq \R\}$ of those vertices $v$ for
  which $B(v)$ is bounded in the $i^{\text{th}}$ coordinate. 
  Then $V_1, \ldots, V_d$ is a
  partition of $V(H)$ and for each $i \in [d]$ the set
  $\{I_i(v) \mid v \in V_i\}$ is an intersection representation with
  intervals of some graph $G_i$ with vertex set $V_i$.
  Then we have $H = \bar{G}_1 \cap \cdots \cap \bar{G}_d$ and hence
  $H^c = \bar{G}_1^c \cup \cdots \cup \bar{G}_d^c = G_1^c \cupdot \cdots \cupdot G_d^c$.
  Thus $H^c$ is the vertex-disjoint union of the $d$ co-interval graphs, as desired.

 \medskip
 
 Now let $H^c = G^c_1 \cupdot \cdots \cupdot G^c_d$, where $G^c_i \in \calC$ for $i = 1,\ldots,d$.
 Consider for each $i$ an intersection representation $\{I_i(v) \mid v \in V(G_i)\}$ of the complement $G_i$ of $G_i^c$ with intervals.
 For $v \in V(H)$ we define 
 \[
  I'_i(v) = \begin{cases} I_i(v), & \text{ if } v \in V(G_i)\\
            \R, & \text{ if } v \notin V(G_i).
           \end{cases}
 \]
 Then $B(v) = I'_1(v) \times \cdots \times I'_d(v)$ is a $1$-local $d$-dimensional box.
 Moreover, $\{B(v) \mid v \in V(H)\}$ is an intersection representation of $H$, which concludes the proof.
\end{proof}

From Lemma~\ref{lem:union-interpretation} we easily derive Theorem~\ref{thm:geometric-interpretation}, i.e., the geometric intersection representations characterizing the local and union boxicity, respectively.

\begin{figure}[tb]
  \centering
  \includegraphics[page=2]{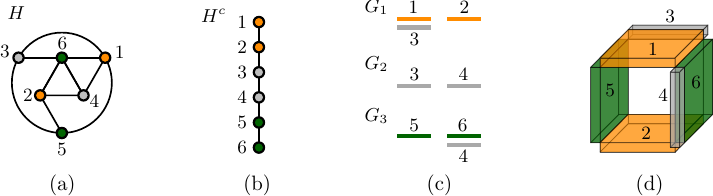}
  \caption{(a, b) A graph $H$ and its complement $H^c$.  (c)~$H^c$ can
    be covered using three co-interval graphs.  (d)~The resulting
    intersection representation.  Note that the boxes are
    3-dimensional as the cover uses three co-interval graphs and the
    boxes are 1-local and 2-local if the corresponding vertices are
    covered once ($1, 2, 5, 6$) and twice ($3, 4$), respectively.
    The long sides of each box have actually infinite
    length.}
  \label{fig:geometric-representation-2}
\end{figure}

%
%
\begin{proof}[Proof of Theorem~\ref{thm:geometric-interpretation}]{\ }
 \begin{enumerate}[label = (\roman*)]
  \item This follows easily from Lemma~\ref{lem:union-interpretation}.
   Indeed, if $\barbo(H) = c_g^{\barC}(H^c) \leq k$, then $H^c = G_1 \cup \cdots \cup G_k$ where for $i = 1,\ldots,k$ the graph $G_i \in \barC$ is the vertex-disjoint union of some $d_i$ co-interval graphs.
   By Lemma~\ref{lem:union-interpretation} $G_i^c$ has an intersection representation with $1$-local $d_i$-dimensional boxes.
   Similarly to the proof of Lemma~\ref{lem:union-interpretation}, extending this $1$-local box representation of $G_i^c$ to all vertices of $H$ by adding a box $\R^{d_i}$ for each vertex in $H - G_i$, and taking the Cartesian product of these $k$ extended $1$-local box representations, we obtain an intersection representation of $H$ of the desired kind.
   
   Similarly, consider any intersection representation $\{B_1(v) \times \cdots \times B_k(v) \mid v \in V(H)\}$ of $H$, where for every $v \in V(H)$ and every $i \in [k]$ the box $B_i(v)$ is $d_i$-dimensional and $1$-local.
   Then by Lemma~\ref{lem:union-interpretation} the set $\{B_i(v) \mid v \in V(H)\}$ is an intersection representation of some graph $G_i$ whose complement $G_i^c$ is in $\barC$.
   Moreover, $H^c$ is the union of these $k$ graph $G_1^c,\ldots,G_k^c \in \barC$.
   This gives $\barbo(H) = c_g^{\barC}(H^c) \leq k$, as desired.
 
 \item For an example illustrating this case, see
   Figure~\ref{fig:geometric-representation-2}.  If
   $\bo_\ell(H) = c_\ell^{\calC}(H^c) \leq k$, then there is a set
   $\{G_1,\ldots,G_t\}$ of $t$ co-interval graphs such that
   $G_i \subseteq H^c$ for $i=1,\ldots,t$,
   $E(H^c) = E(G_1) \cup \cdots \cup E(G_t)$ and every $v \in V(H^c)$
   is contained in at most $k$ such $G_i$, $i = 1,\ldots,t$.  For each
   $i \in [t]$ consider an interval representation
   $\{I_i(v) \mid v \in V(G_i)\}$ of $G_i^c$.  For $v \in H - G_i$ we
   set $I_i(v) = \R$.  Note that $\{I_i(v) \mid v\in V(H)\}$ is an
   interval representation of $\bar{G}_i^c$.
  

  Now for $v \in V(G)$ let $B(v) = I_1(v) \times \cdots \times I_t(v)$ be the Cartesian product of the $t$ intervals associated with vertex $v$.
  As $v$ is in $G_i$ for at most $k$ indices $i \in [t]$, $I_i(v) \neq \R$ for at most $k$ indices $i \in [t]$.
  In other words, $B(v)$ is a $k$-local box.
  Finally, we claim that $\{B(v) \mid v \in V(H)\}$ is an intersection representation of $H$.
  Indeed, if $vw \notin E(H)$, then $vw \in E(H^c)$ and hence $vw \in E(G_i)$ for at least one $i \in [t]$.
  Then $I_i(v) \cap I_i(w) = \emptyset$ and thus $B(v) \cap B(w) = \emptyset$.
  And if $vw \in E(H)$, then $vw \notin E(H^c)$ and $vw \notin E(G'_i)$ for every $i \in [t]$.
  Thus $I_i(v) \cap I_i(w) \neq \emptyset$ for every $i \in [t]$ and hence $B(v) \cap B(w) \neq \emptyset$.
  
  This shows that if $\bo_\ell(H) \leq k$, then $H$ is the intersection graph of $k$-local boxes.
  On the other hand, if $H$ admits an intersection representation with $k$-local $t$-dimensional boxes, then for each $i \in [t]$ projecting the boxes to coordinate $i$ and considering the bounded intervals in this projection gives an interval representation of some subgraph $G_i$ of $H^c$.
  As before, we can check that $\{G_1,\ldots,G_t\}$ forms an injective $k$-local $\calC$-cover of $H^c$, showing that $\bo_\ell(H) = c_\ell^{\calC}(H^c) \leq k$.\qedhere
 \end{enumerate} 
\end{proof}

\section{Conclusions}\label{sec:conclusions}

In this paper we have introduced the notions of the local boxicity $\bo_\ell(H)$ and union boxicity $\barbo(H)$ of a graph $H$.
It holds that $\bo_\ell(H) \leq \barbo(H) \leq \bo(H)$, where $\bo(H)$ denotes the classical boxicity as introduced  almost 50 years ago.
Indeed, both new parameters are a better measure of the complexity of $H$.
For example, if $H$ is the complement of a matching on $n$ edges, then $\bo(H) = n$, simply because the $n$ non-edges each have to be realized in a different dimension.
On the other hand, we have $\bo_\ell(H) = \barbo(H) = 1$, and as these non-edges are vertex-disjoint, they also should be ``counted only once''.
We have shown this phenomenon in a few more examples in the course of the paper.
In fact, in many box representations from the literature many (if not all) dimensions are only used by few vertices.
The resulting high boxicity may be misintepreted as the graph being very complex, which could be avoided by using local or union boxicity.

In future research, established boxicity results should be revisited to see whether one can improve the upper bounds using local or union boxicity.
For example, it is known that if $H$ is a planar graph, then
$\bo(H) \leq 3$~\cite{Tho-86}. Moreover, the octahedral graph $O$ is
planar and has boxicity~$3$, because its complement $O^c$ is the
matching on three edges (c.f. the proof of
Theorem~\ref{thm:separation}~\ref{enum:union-smaller-global} and
Figure~\ref{fig:geometric-interpretation}).
By~\eqref{eq:inequalities} we have that
$\bo_\ell(H) \leq \barbo(H) \leq 3$ whenever $H$ is planar.  However,
$\bo_\ell(O) = \barbo(O) = 1$, because $O^c$ is the vertex-disjoint
union of co-interval graphs, i.e., $O^c \in \barC$.  Hence it is
natural to ask the following.

\begin{question}
 Is there a planar graph $H$ with $\bo_\ell(H) = 3$?
\end{question}

For general graphs $H$ we proved that the local boxicity $\bo_\ell(H)$ and the union boxicity $\barbo(H)$ can be arbitrarily far from the classical boxicity $\bo(H)$.
But we do not know whether if $\bo(H)$ is large, then $\bo_\ell(H)$ and $\barbo(H)$ can be very \emph{close} to $\bo(H)$.
We construct graphs in the proof of Theorem~\ref{thm:separation}~\ref{enum:local-high} with large local boxicity, but one can show that these have even larger boxicity.

\begin{question}
 Is there for every $k \in \N$ a graph $H_k$ such that $\bo_\ell(H_k) = \barbo(H_k) = \bo(H_k) = k$?
\end{question}

Another interesting research direction concerns the computational complexity.
It is known that for every $k \geq 2$ deciding whether a given graph $H$ satisfies $\bo(H) \leq k$ is NP-complete~\cite{Coz-81,Kra-94}.
For $k = 1$ we have $\bo(H) \leq k$ if and only if $H$ is an interval graph, and $\barbo(H) \leq k$ (equivalently $\bo_\ell(H) \leq k$) if and only if the complement of $H$ is the vertex-disjoint union of co-interval graphs, both of which can be tested in polynomial time via interval graph recognition~\cite{Boo-76}.

\begin{question}
 For $k \geq 2$, is it NP-complete to decide whether $\bo_\ell(H) \leq k$ (or $\barbo(H) \leq k$) for a given graph $H$?
\end{question}

Let us remark that for general covering numbers the computational complexity of computing $c_g^{\calG}(H)$ tends to be harder than that of $c_\ell^{\calG}(H)$, which in turn tends to be harder than for $c_f^{\calG}(H)$.
For example, for $\calG$ being the class of star forests, computing $c_g^{\calG}(H)$ is NP-complete~\cite{Gon-09,Hak-96}, while computing $c_\ell^{\calG}(H)$ and $c_f^{\calG}(H)$ is polynomial-time solvable~\cite{KU16}.
The same holds when $\calG$ is the class of all matchings as discussed in~\cite{KU16}.
And for $\calG$ being the class of bipartite graphs, computing $c_g^{\calG}(H)$ and $c_\ell^{\calG}(H)$ is NP-complete~\cite{Fis-96}, while computing $c_f^{\calG}(H)$ is polynomial-time solvable since $c_f^{\calG}(H) = 1$ if $H$ is bipartite and $c_f^{\calG}(H) = 2$ otherwise.

\bibliography{lit}
\bibliographystyle{amsplain}

\end{document}